\documentclass[12pt,a4paper]{amsart}
\usepackage{indentfirst}
\usepackage[iso-8859-7]{inputenc}
\usepackage{graphics}
\usepackage{graphicx}
\usepackage{undertilde}
\usepackage{amssymb}
\usepackage{amsmath}
\usepackage{amsthm}
\usepackage{latexsym}
\usepackage{verbatim}
\usepackage{mathrsfs}
\usepackage{hyperref}
\makeatletter
\def\verbatim@font{\usefont{OT1}{cmtt}{m}{n}}
\makeatother

\usepackage{color}

\newsavebox{\ChapterTitle}
\makeatletter
\newcommand{\ps@spawn}{%
    \renewcommand{\@oddhead}{{\usebox{\ChapterTitle}\hfil\thepage}\hspace{-13.83cm}\rule[-8pt]{13.83cm}{0.5pt}}%
    \renewcommand{\@evenhead}{{\thepage\hfil {\selectlanguage{english} Chapter} \thechapter}\hspace{-13.83cm}\rule[-8pt]{13.83cm}{0.5pt}}%
    \renewcommand{\@evenfoot}{}%
    \renewcommand{\@oddfoot}{\@evenfoot}}
\makeatother

\definecolor{mygray}{gray}{0.9}

\newcommand{\ci}[1]{_{ {}{\scriptstyle #1}}}

\newtheorem{thm}{Theorem}[section]
\newtheorem{prop}[thm]{Proposition}
\newtheorem{lemma}[thm]{Lemma}

\theoremstyle{remark}
\newtheorem{remark}[thm]{Remark}

\numberwithin{equation}{section}

\begin{document}

\title{A note on a two-weight estimate for the dyadic square function}
\author{Spyridon Kakaroumpas}
\date{}
\maketitle

\begin{abstract}
We show that the two-weight estimate for the dyadic square function proved by Lacey--Li in \cite{Lacey-Li} is sharp.
\end{abstract}

\section{Introduction}

This note deals with two-weight estimates for dyadic square functions. More precisely, fix a positive integer $d$. We will say that a subset $Q$ of $\mathbb{R}^d$ is a cube (interval for the case $d=1$) if there exist $a_1,b_1,\ldots,a_d,b_d\in\mathbb{R}$ with $a_i<b_i,\;i=1,\ldots,d$, such that $Q=[a_1,b_1)\times\ldots\times[a_d,b_d)$. Then, a subset $Q$ of $\mathbb{R}^d$ is called a dyadic cube if there exist $n,k_1,\ldots,k_d\in\mathbb{Z}$ such that
\begin{equation*}
Q=\left[\frac{k_1}{2^n},\frac{k_1+1}{2^n}\right)\times\ldots\times\left[\frac{k_d}{2^n},\frac{k_d+1}{2^n}\right).
\end{equation*}
Let $Q_0$ be either $\mathbb{R}^d$ or a fixed dyadic cube of $\mathbb{R}^d$. Let $\mathcal{D}$ be the family of all dyadic cubes contained in $Q_0$. For all $Q\in\mathcal{D}$, we denote by $\text{ch}(Q)$ the family of all dyadic children of $Q$, i.e. the family of all maximal dyadic cubes of $\mathbb{R}^d$ contained in $Q$. Then, for all $f\in L^1_{\text{loc}}(Q_0)$ (locally integrable means here that $f$ is integrable over all cubes contained in $Q_0$), for all $Q\in\mathcal{D}$, we consider the function $\Delta_{\ci{Q}}f$ on $Q_0$ given by
\begin{equation*}
\Delta_{\ci{Q}}f=\sum_{R\in\text{ch}(Q)}\langle f\rangle_{\ci{R}}1_{\ci{R}}-\langle f\rangle_{\ci{Q}}1_{\ci{Q}},
\end{equation*}
where for all cubes $R$ in $\mathbb{R}^d$, we denote by $1_{\ci{R}}$ the characteristic function of $R$ and we also set $\langle f\rangle_{\ci{R}}=\frac{1}{|R|}\int_{R}f(x)dx$ ($dx$ and $|\cdot|$ denoting usual Lebesgue measure on $\mathbb{R}^d$). Then, for all $f\in L^1_{\text{loc}}(Q_0)$, we define the dyadic square function $Sf$ on $Q_0$ by
\begin{equation*}
Sf(x)=\left(\sum_{Q\in\mathcal{D}}(\Delta_{\ci{Q}}f(x))^2\right)^{\frac{1}{2}},\;\forall x\in Q_0.
\end{equation*}
A measurable function $\sigma$ on $Q_0$ is called a weight if it is positive a.e. on $Q_0$ and locally integrable over $Q_0$. Given a weight $\sigma$ on $Q_0$, one can define as in \cite{Lacey-Li} p. 2 the $A_{\infty}$ characteristics
\begin{equation*}
[\sigma]_{\ci{A_{\infty}}}:=\sup_{Q\text{ cube contained in }Q_0}\frac{1}{\sigma(Q)}\int_{Q}M_{\ci{Q}}\sigma(x)dx,
\end{equation*}
\begin{equation*}
[\sigma]_{\ci{A_{\infty},\mathcal{D}}}:=\sup_{Q\in\mathcal{D}}\frac{1}{\sigma(Q)}\int_{Q}M_{\ci{Q,\mathcal{D}}}\sigma(x)dx,
\end{equation*}
where
\begin{equation*}
M_{\ci{Q}}\sigma (x):=
\sup_{\substack{R\text{ cube contained in }Q\\x\in R}}\langle\sigma\rangle_{\ci{R}},\;\forall x\in Q_0,
\end{equation*}
for all cubes $Q$ contained in $Q_0$, and
\begin{equation*}
M_{\ci{Q,\mathcal{D}}}\sigma (x):=
\sup_{\substack{R\in\mathcal{D},\; R\subseteq Q\\x\in R}}\langle\sigma\rangle_{\ci{R}},\;\forall x\in Q_0,
\end{equation*}
for all $Q\in\mathcal{D}$ (suprema over empty sets are set to be equal to 0). Moreover, given weights $\sigma,w$ on $Q_0$ and $p\in(1,\infty)$, we define as in \cite{Lacey-Li} p. 2 the joint $A_p$ characteristics
\begin{equation*}
[w,\sigma]_{\ci{A_{p}}}:=\sup_{Q\text{ cube contained in }Q_0}\langle w\rangle_{\ci{Q}}\langle \sigma\rangle_{\ci{Q}}^{p-1},
\end{equation*}
\begin{equation*}
[w,\sigma]_{\ci{A_{p},\mathcal{D}}}:=\sup_{Q\in\mathcal{D}}\langle w\rangle_{\ci{Q}}\langle \sigma\rangle_{\ci{Q}}^{p-1},
\end{equation*}
and we also define the Muckenhoupt $A_{p}$ characteristics
\begin{equation*}
[w]_{\ci{A_{p}}}:=[w,w^{-\frac{1}{p-1}}]_{\ci{A_{p}}},
\end{equation*}
\begin{equation*}
[w,\sigma]_{\ci{A_{p},\mathcal{D}}}:=[w,w^{-\frac{1}{p-1}}]_{\ci{A_{p},\mathcal{D}}}.
\end{equation*}
It is proved in \cite{Convex} Lemma 4.1. that for all weights $w$ on $Q_0$ and for all $p\in(1,\infty)$ there holds $1\leq[\sigma]_{\ci{A_{\infty},\mathcal{D}}}\lesssim_{p} [\sigma]_{\ci{A_{p},\mathcal{D}}}$ and $1\leq[\sigma]_{\ci{A_{\infty}}}\lesssim_{d,p} [\sigma]_{\ci{A_{p}}}$.
Finally, for all weights $\sigma$ on $Q_0$, we consider the space $L^{p}(\sigma):=L^{p}(Q_0,\sigma(x)dx)$. The following estimate is then proved in \cite{Lacey-Li} using a sparse domination argument similar to the one in \cite{Laceymartingale} and a two-weight estimate for square sparse functions.

\begin{thm}\emph{(\cite{Lacey-Li} Theorem 1.2.)}
For all weights $\sigma,w$ on $Q_0$, for all $p\in(1,\infty)$, there holds
\begin{equation*}
||S(\cdot\sigma)||_{\ci{L^p(\sigma)\rightarrow L^{p}(w)}}\lesssim_{p,n}
\begin{cases}
[w,\sigma]_{\ci{A_{p},\mathcal{D}}}^{\frac{1}{p}}[\sigma]_{\ci{A_{\infty},\mathcal{D}}}^{\frac{1}{p}},\text{ if }p\leq 2\\
[w,\sigma]_{\ci{A_{p},\mathcal{D}}}^{\frac{1}{p}}([w]_{\ci{A_{\infty},\mathcal{D}}}^{\frac{1}{2}-\frac{1}{p}}+[\sigma]_{\ci{A_{\infty},\mathcal{D}}}^{\frac{1}{p}}),\text{ if }p>2
\end{cases}
.
\end{equation*}
\end{thm}

Note that the above estimate is not symmetric with respect to the weights $\sigma$ and $w$. One should not expect a priori any symmetry between the two weights, because the square function is not a linear operator, and even when viewed as a vector-valued linear operator, its adjoint is not of the ``same form'' (in contrast to the case of Calder$\acute{\text{o}}$n--Zygmund operators, which are linear and their adjoint is also a Calder$\acute{\text{o}}$n--Zygmund operator; see \cite{Hyt-Lacey} Theorem 1.2.). One may then ask whether the presence of some term involving $[w]_{\ci{A_{\infty},\mathcal{D}}}$ in the above  estimate for the case $p>2$ is really necessary, and not just an artifact of the proof given in \cite{Lacey-Li}. The following result answers this to the positive.

\begin{prop}\label{countprop}
For all $p\in(2,\infty)$, there exist weights $w,\sigma$ on $\mathbb{R}$ with $[\sigma]_{\ci{A_{\infty}}}<\infty$, $[w,\sigma]_{\ci{A_{p}}}<\infty$ and $f\in L^{p}(\sigma)$ with $\text{supp}(f)\subseteq[0,1]$, such that $S(f\sigma)\notin L^{p}(w)$ (and $S$ can be taken to be just the square function over $[0,1)$).
\end{prop}

More generally, one may ask whether the exponent $\frac{1}{p}$ for $[\sigma]_{\ci{A_{\infty},\mathcal{D}}}$ and the exponent $\frac{1}{2}-\frac{1}{p}$ for $[w]_{\ci{A_{\infty},\mathcal{D}}}$ (in the case $p>2$) are optimal. It turns out that this is the case. More precisely, we prove the following.

\begin{prop}\label{expprop}
(i) Let $p\in(1,2]$. Let $\psi:[1,\infty)\rightarrow[1,\infty)$ be an increasing function with $\psi(1)=1$, such that for all $d\in\mathbb{N}\setminus\lbrace0\rbrace$ and for all weights $w,\sigma$ on $\mathbb{R}^{d}$ with $[\sigma]_{\ci{A_{\infty}}}<\infty$ there holds
\begin{equation*}
||S(\cdot\sigma)||_{\ci{L^p(\sigma)\rightarrow L^{p}(w)}}\lesssim_{p,d}
[w,\sigma]_{\ci{A_{p}}}^{\frac{1}{p}}\psi([\sigma]_{\ci{A_{\infty}}}).
\end{equation*}
Then, there exists an absolute constant $A>0$ with $\psi(x)\gtrsim_{p}x^{\frac{1}{p}}$, for all $x>A$.

(ii) Let $p\in(2,\infty)$. Let $\phi,\psi:[1,\infty)\rightarrow[1,\infty)$ be increasing functions with $\phi(1)=\psi(1)=1$ (and depending only on $p$), such that for all $d\in\mathbb{N}\setminus\lbrace0\rbrace$ and for all weights $w,\sigma$ on $\mathbb{R}^{d}$ with $[\sigma]_{\ci{A_{\infty}}}<\infty$ and $[w]_{\ci{A_{\infty}}}<\infty$ there holds
\begin{equation*}
||S(\cdot\sigma)||_{\ci{L^p(\sigma)\rightarrow L^{p}(w)}}\lesssim_{p,d}
[w,\sigma]_{\ci{A_{p}}}^{\frac{1}{p}}(\phi([w]_{\ci{A_{\infty}}})+\psi([\sigma]_{\ci{A_{\infty}}})).
\end{equation*}
Then, there exist absolute constants $A,B>0$ with $\psi(x)\gtrsim_{p}x^{\frac{1}{p}}$, for all $x>A$ and $\phi(x)\gtrsim_{p}x^{\frac{1}{2}-\frac{1}{p}}$, for all $x>B$.
\end{prop}

\textbf{Plan of the note.} In Section 2. we give the estimates of Muckenhoupt characteristics that will be important for our counterexamples. In Section 3., we prove (i) and the first half of (ii) of \hyperref[expprop]{Proposition \ref*{expprop}}, using the family of examples introduced by Lerner in \cite{Lerner} Section 5., (5.5). In Section 4., we prove the second half of (ii) of \hyperref[expprop]{Proposition \ref*{expprop}}, using a family of examples reminiscent of a counterexample for two-weight estimates for square functions in non-homogeneous settings due to Lai--Treil in \cite{Lai-Treil} Section 3. However, this family of examples does not directly yield an explicit counterexample as the one in \hyperref[countprop]{Proposition \ref*{countprop}}. Therefore, in Section 5. we present two explicit counterexamples of this kind. The first of these can be viewed as a modified version of a direct sum construction over the family of examples in Section 4. I am grateful to Professor S. Treil for suggesting the direct sum construction. The second counterexample follows very closely the strategy of Lai--Treil in \cite{Lai-Treil} Section 3.

\textbf{Acknowledgements.}
I am grateful to Professor S. Treil for suggesting the problem of this note to me and for his suggestions on improvements of various aspects of this note, leading in particular to more transparent estimates and to a significant simplification of the computations in the counterexample in Subsection 5.2. I am also grateful to Alexander Barron for reading a draft of this note and for pointing out typos and other obscurities.
\newline

\section{Preliminary estimates}

We will need two lemmas. The first one concerns Muckenhoupt characteristics of power weights.

\begin{lemma}\label{prelest} Let $\beta\in(0,1)$ and $p\in(1,\infty)$.

(i) Consider the weights $w,\sigma$ on $\mathbb{R}$ given by $w(x)=|x|^{-\beta}$ and $\sigma(x)=|x|^{\frac{\beta}{p-1}}$, for all $x\in\mathbb{R}\setminus\lbrace0\rbrace$, and $\sigma(0)=w(0)=1$. Then, there holds $[w]_{\ci{A_{p}}}=[w,\sigma]_{\ci{A_{p}}}\sim_{p}(1-\beta)^{-1}$. In particular, we have $[w]_{\ci{A_{\infty}}}\lesssim(1-\beta)^{-1}$ and $[\sigma]_{\ci{A_{\infty}}}\lesssim_{p}1$.

(ii) Consider the weights $w,\sigma$ on $\mathbb{R}$ given by $\sigma(x)=|x|^{-\beta}$ and $w(x)=|x|^{\beta(p-1)}$, for all $x\in\mathbb{R}\setminus\lbrace0\rbrace$, and $\sigma(0)=w(0)=1$. Then, there holds $[w,\sigma]_{\ci{A_{p}}}\sim_{p}(1-\beta)^{1-p}$. In particular, we have $[w]_{\ci{A_{\infty}}}\lesssim_{p}1$.
\end{lemma}

\begin{proof}
(i) We first consider the cases of subintervals of $[0,\infty)$. Let $0\leq a<b$ be arbitrary. We distinguish the following cases:

(A) $a=0$. Then, by direct computation we have
\begin{eqnarray*}
\langle w\rangle_{[0,b)}\langle \sigma\rangle_{[0,b)}^{p-1}=(1-\beta)^{-1}\left(1+\frac{\beta}{p-1}\right)^{1-p}.
\end{eqnarray*}

It is easy to verify that $\frac{1}{e}\leq\left(1+\frac{\beta}{p-1}\right)^{1-p}\leq1$.

(B) $0<a<\frac{b}{2}$. Then, we have $[a,b)\subseteq [0,b)$ and $b-a\geq\frac{b}{2}$, therefore by case (A) we obtain
\begin{eqnarray*}
\langle w\rangle_{[a,b)}\langle\sigma\rangle_{[a,b)}^{p-1}\lesssim_{p}\langle w\rangle_{[0,b)}\langle\sigma\rangle_{[0,b)}^{p-1}\sim(1-\beta)^{-1}.
\end{eqnarray*}

(C) $a\geq\frac{b}{2}$. Then, we have $\frac{1}{2}w(a)\leq w(x)\leq w(a)$ and $\sigma(a)\leq\sigma(x)\leq 2^{\frac{1}{p-1}}\sigma(a)$, for all $x\in[a,b)$, therefore
\begin{eqnarray*}
\langle w\rangle_{(a,b)}\langle\sigma\rangle_{(a,b)}^{p-1}\sim w(a)\sigma(a)^{p-1}=1.
\end{eqnarray*}

Similarly we obtain $\langle w\rangle_{[a,b)}\langle\sigma\rangle_{[a,b)}^{p-1}\lesssim_{p} (1-\beta)^{-1}$, for all subintervals $I$ of $(-\infty,0)$.

Let now $a,b\in(0,\infty)$ be arbitrary. We will show that $\langle w\rangle_{[-a,b)}\langle\sigma\rangle_{[-a,b)}^{p-1}\lesssim_{p} (1-\beta)^{-1}$. Set $c=\max(a,b)$. We have
\begin{align*}
\langle w\rangle_{[-a,b)}\langle\sigma\rangle_{[-a,b)}^{p-1}
&\lesssim_{p}\frac{(w([-c,0))+w([0,c)))(\sigma([-c,0))^{p-1}+\sigma([0,c))^{p-1})}{(b+a)^{p}}\\
&\lesssim(1-\beta)^{-1}\frac{c^{p}}{(b+a)^{p}}\leq(1-\beta)^{-1}.
\end{align*}
This yields the desired conclusion for $[w,\sigma]_{\ci{A_{p}}}$. Note that the proof shows also that $\frac{(1-\beta)^{-1}}{e}\leq[w,\sigma]_{\ci{A_{p},\mathcal{D}}}\leq\max((1-\beta)^{-1},2)$.

We have $[w]_{\ci{A_{\infty}}}\lesssim[w]_{\ci{A_{2}}}\lesssim(1-\beta)^{-1}$. Set now $q=\frac{p+1}{2}$. Then, we have $q>1$ and $\delta=\frac{\beta}{p-1}(q-1)=\frac{\beta}{2}\in(0,1)$. Consider the weight $\rho$ on $\mathbb{R}$ given by $\rho(x)=|x|^{-\delta}$, for all $x\in\mathbb{R}\setminus\lbrace0\rbrace$ and $\rho(0)=1$. Then, we have $\sigma(x)=|x|^{\frac{\delta}{q-1}}$, for all $x\in\mathbb{R}\setminus\lbrace0\rbrace$, and therefore by the above we obtain
\begin{equation*}
[\sigma]_{\ci{A_{\infty}}}\lesssim_{q'}[\sigma]_{\ci{A_{q'}}}=[\rho,\sigma]_{\ci{A_{q}}}^{\frac{1}{q-1}}\lesssim_{q}\left(1-\frac{\beta}{2}\right)^{-\frac{1}{q-1}}\leq 2^{\frac{1}{q-1}},
\end{equation*}
therefore $[\sigma]_{\ci{A_{\infty}}}\lesssim_{p}1$.

(ii) Since $w=\sigma^{-\frac{1}{p'-1}}$, by (i) we have
\begin{equation*}
[w,\sigma]_{\ci{A_{p}}}=[\sigma,w]_{\ci{A_{p'}}}^{\frac{1}{p'-1}}=[\sigma,w]_{\ci{A_{p'}}}^{p-1}\lesssim_{p'} (1-\beta)^{1-p},
\end{equation*}
therefore $[w,\sigma]_{\ci{A_{p}}}\lesssim_{p}(1-\beta)^{1-p}$, and also $[w]_{\ci{A_{\infty}}}\lesssim_{p'}1$, i.e. $[w]_{\ci{A_{\infty}}}\lesssim_{p}1$.
\newline

\end{proof}
Power weights are the prototypes of weights having just one singularity. The second lemma shows that under certain conditions, one can extend such weights from $[0,1)$ to the whole real line in a way that joint Muckenhoupt characteristics are not destroyed.

\begin{lemma}\label{glob}
Let $p\in(1,\infty)$. Set $I_{k}=\left[0,\frac{1}{2^{k}}\right)$, for all $k=0,1,2,\ldots$. Let $w,\sigma$ be weights on $[0,1)$, such that the following hold:

(i) $[w,\sigma]_{\ci{A_{p}}}\lesssim_{p}1$.

(ii) $w(I_{k})\sigma(I_{l})^{p-1}\lesssim_{p}(|I_{k}|+|I_{l}|)^{p}$, for all $k,l=0,1,2,\ldots$

(iii) There exists $x_{0}\in\left[\frac{1}{2},1\right)$ (depending only on $p$), such that $w(x)\lesssim_{p}1$, for all $x\in[x_0,1)$ and $\sigma(x)\lesssim_{p}1$, for all $x\in[x_0,1)$.

Consider the weights $\widetilde{w},\widetilde{\sigma}$ on $\mathbb{R}$ given by
\begin{equation*}
\widetilde{w}(x)=
\begin{cases}
w(x-k+1),\;\forall x\in(k-1,k),\text{ if }k\text{ is odd}\\
w(k-x),\;\forall x\in(k-1,k),\text{ if }k\text{ is even}
\end{cases}
,\;\forall k\in\mathbb{Z},
\end{equation*}
\begin{equation*}
\widetilde{\sigma}(x)=
\begin{cases}
\sigma(x-k+1),\;\forall x\in(k-1,k),\text{ if }k\text{ is odd}\\
\sigma(k-x),\;\forall x\in(k-1,k),\text{ if }k\text{ is even}
\end{cases}
,\;\forall k\in\mathbb{Z},
\end{equation*}
\begin{equation*}
\widetilde{w}(k)=\widetilde{\sigma}(k)=1,\;\forall k\in\mathbb{Z}.
\end{equation*}
Then, there holds $[\widetilde{w},\widetilde{\sigma}]_{\ci{A_{p}}}\lesssim_{p}1$.
\end{lemma}
\begin{proof}
First of all, note that for all $a,b\in(0,1]$, there exist natural numbers $k,l$ with $\frac{1}{2^{k+1}}<a\leq\frac{1}{2^{k}}$ and $\frac{1}{2^{l+1}}<b\leq\frac{1}{2^{l}}$, therefore by assumption we have
\begin{align}
\frac{w([0,a))\sigma([0,b))^{p-1}}{(b+a)^{p}}&\lesssim_{p}
\frac{w(I_{k})\sigma(I_{l})^{p-1}}{(2^{-k}+2^{-l})^{p}}\lesssim_{p}1.
\end{align}
We next notice that for all $k,m\in\mathbb{Z}$ with $k<m$, there holds
\begin{equation*}
\langle \widetilde{w}\rangle_{[k,m)}\langle\widetilde{\sigma}\rangle_{[k,m)}^{p-1}=
\frac{(m-k+1)w([0,1))(m-k+1)^{p-1}\sigma([0,1))^{p-1}}{(m-k+1)^{p}}\lesssim_{p}1.
\end{equation*}
Let now $a,b\in\mathbb{R}$ with $a<b$ be arbitrary. If $[a,b)\subseteq [k,k+1)$ for some $k\in\mathbb{Z}$, then it is clear that
\begin{equation*}
\langle \widetilde{w}\rangle_{[a,b)}\langle\widetilde{\sigma}\rangle_{[a,b)}^{p-1}\leq[w,\sigma]_{\ci{A_{p}}}\lesssim_{p}1.
\end{equation*}
Assume now that $[a,b)\nsubseteq[k,k+1)$, for all $k\in\mathbb{Z}$. There exist then $k,m\in\mathbb{Z}$ with $k\leq a<k+1\leq m\leq b<m+1$. If $k+1<m$, then $m+1-k\geq 3$, therefore
\begin{equation*}
b-a\geq m-(k+1)=m+1-k-2\geq \frac{m+1-k}{4},
\end{equation*}
thus since $[a,b)\subseteq[k,m+1)$ we obtain
\begin{equation*}
\langle \widetilde{w}\rangle_{[a,b)}\langle\widetilde{\sigma}\rangle_{[a,b)}^{p-1}\lesssim_{p}\langle \widetilde{w}\rangle_{[k,m+1)}\langle\widetilde{\sigma}\rangle_{[k,m+1)}^{p-1}\lesssim_{p}1.
\end{equation*}
Assume now that $m=k+1$. Then $b>k+1$. If $a\leq k+x_0$ or $b\geq k+1+x_0$, then $b-a\geq\min(x_0,1-x_0)=1-x_0$, therefore since $[a,b)\subseteq[k,k+2)$ we obtain
\begin{equation*}
\langle \widetilde{w}\rangle_{[a,b)}\langle\widetilde{\sigma}\rangle_{[a,b)}^{p-1}\lesssim_{p}\langle \widetilde{w}\rangle_{[k,k+2)}\langle\widetilde{\sigma}\rangle_{[k,k+2)}^{p-1}\lesssim_{p}1.
\end{equation*}
Assume now that $a>k+x_0$ and that $b<k+1+x_0$. We distinguish the following cases:

(A) The integer $k+1$ is odd. Then, we have $\widetilde{w}(x)=w(x-k)\lesssim_{p}1$, for all $x\in[a,b)$ and $\widetilde{\sigma}(x)=\sigma(k+2-x)\lesssim_{p}1$, for all $x\in[a,b)$, therefore $\langle \widetilde{w}\rangle_{[a,b)}\langle\widetilde{\sigma}\rangle_{[a,b)}^{p-1}\lesssim_{p}1$.

(B) The integer $k+1$ is even. Set $a'=k+1-a$ and $b'=b-(k+1)$. Then, by (2.1) we have
\begin{align*}
\langle \widetilde{w}\rangle_{[a,b)}\langle\widetilde{\sigma}\rangle_{[a,b)}^{p-1}
&\lesssim_{p}
\frac{(w([0,a'))+w([0,b')))(\sigma([0,a'))^{p-1}+\sigma([0,b'))^{p-1})}{(a'+b')^{p}}\lesssim_{p}1.
\end{align*}
This concludes the proof.
\newline
\end{proof}

\section{The exponent of $[\sigma]_{\ci{A_{\infty}}}$}

In this section we prove part (i) and the first half of part (ii) of \hyperref[expprop]{Proposition \ref*{expprop}}. Let $p\in(1,\infty)$. As mentioned in the introduction, we will use the family of examples introduced by Lerner in \cite{Lerner} Section 5., (5.5). Let $\beta\in\left(\frac{1}{2},1\right)$. Consider the functions $w,\sigma,f$ on $\mathbb{R}$ given by
\begin{equation*}
\sigma(x)=|x|^{-\beta},\;w(x)=|x|^{\beta(p-1)},\;f(x)=1_{[0,1]}(x),\;\forall x\in\mathbb{R}\setminus\lbrace0\rbrace
\end{equation*}
and $f(0)=\sigma(0)=w(0)=1$. 

By \hyperref[prelest]{Lemma \ref*{prelest}} we have $[w,\sigma]_{\ci{A_{p}}}\sim_{p}(1-\beta)^{1-p}$, $[w]_{\ci{A_{\infty}}}\lesssim_{p}1$ and $[\sigma]_{\ci{A_{\infty}}}\leq A(1-\beta)^{-1}$, for some absolute constant $A>0$. We also have
\begin{eqnarray*}
||f||_{L^{p}(\sigma)}^{p}=\int_{[0,1)}x^{-\beta}dx=(1-\beta)^{-1}.
\end{eqnarray*}

\begin{lemma}
There holds $||S(\sigma f)||_{\ci{L^{p}(w)}}\gtrsim_{p}(1-\beta)^{-1-\frac{1}{p}}$.
\end{lemma}
\begin{proof}
Set $I_{n}=\left[0,\frac{1}{2^{n}}\right),\;n=0,1,2,\ldots$ and $J_{n}=I_{n-1}\setminus I_{n}=\left[\frac{1}{2^{n}},\frac{1}{2^{n-1}}\right),\;n=1,2,\ldots$. Let $k\geq0$ be arbitrary. We notice that for all $x\in I_{k+1}$, there holds
\begin{align*}
\Delta_{\ci{I_{k}}}(\sigma f)(x)&=\langle \sigma f\rangle_{\ci{I_{k+1}}}-\langle\sigma f\rangle_{\ci{I_{k}}}=2^{k+1}\int_{\left(0,\frac{1}{2^{k+1}}\right)}x^{-\beta}dx-2^{k}\int_{\left(0,\frac{1}{2^{k}}\right)}x^{-\beta}dx\\
&=(1-\beta)^{-1}(2^{\beta(k+1)}-2^{\beta k})\sim(1-\beta)^{-1}2^{\beta k}.
\end{align*}
In particular, for all $k\geq 0$, for all $n>k$ and for all $x\in I_{n}$, we have $|\Delta_{\ci{I_{k}}}(\sigma f)(x)|\sim(1-\beta)^{-1}2^{\beta k}$. It follows that for all $n\geq1$, for all $x\in J_{n}$, there holds
\begin{align*}
S(\sigma f)(x)^{p}&\geq\left(\sum_{k=0}^{n-1}(\Delta_{\ci{I_{k}}}(\sigma f)(x))^2\right)^{\frac{p}{2}}\sim_{p}(1-\beta)^{-p}\left(\sum_{k=0}^{n-1}2^{2\beta k}\right)^{\frac{p}{2}}\\
&=(1-\beta)^{-p}\frac{(2^{2\beta(n+1)}-1)^{\frac{p}{2}}}{(2^{2\beta}-1)^{\frac{p}{2}}}
\gtrsim (1-\beta)^{-p}2^{\beta np}.
\end{align*}
Moreover, for all $n\geq 1$, we have $w(J_{n})\sim 2^{(-1-\beta(p-1))n}$. It follows that
\begin{equation*}
||S(\sigma f)||_{\ci{L^{p}(w)}}^{p}\gtrsim_{p}(1-\beta)^{-p}\sum_{n=1}^{\infty}2^{(\beta-1)n}\sim
(1-\beta)^{-p}\frac{1}{1-2^{\beta-1}}\sim(1-\beta)^{-p-1},
\end{equation*}
since $\lim_{x\rightarrow 1^{-}}\frac{1-x}{1-2^{x-1}}=\frac{1}{\log 2}\in(0,\infty)$.
\end{proof}

The above estimates yield
\begin{equation*}
(1-\beta)^{-1-\frac{1}{p}}\lesssim_{p}(1-\beta)^{\frac{1}{p}-1}\psi(A(1-\beta)^{-1})(1-\beta)^{-\frac{1}{p}},
\end{equation*}
i.e.
\begin{equation*}
\psi(A(1-\beta)^{-1})\gtrsim_{p} (1-\beta)^{-\frac{1}{p}}.
\end{equation*}
Since $\beta\in\left(\frac{1}{2},1\right)$ was arbitrary, we obtain $\psi(x)\gtrsim_{p}x^{\frac{1}{p}}$, for all $x>2A$.
\newline

\section{The exponent of $[w]_{\ci{A_{\infty}}}$}\label{expw}

In this section we prove the second half of part (ii) of \hyperref[expprop]{Proposition \ref*{expprop}}. As mentioned in the introduction, we will use a family of examples reminiscent of a counterexample for two-weight estimates for square functions in non-homogeneous settings due to Lai--Treil in \cite{Lai-Treil} Section 3. More precisely, this family of examples is reminiscent of the counterexample in section 5.2., which follows in turn very closely the strategy of the construction in \cite{Lai-Treil}, Section 3. 

Let $p\in(2,\infty)$. Consider the functions $f,\sigma,w$ on $\mathbb{R}$ given by
\begin{equation*}
f(x)=(-1)^{\lfloor-\log_2 |x|\rfloor}|x|^{-\frac{\beta}{p-1}}1_{(0,1)}(x),\;\;\sigma(x)=|x|^{\frac{\beta}{p-1}},\;\;w(x)=|x|^{-\beta},\;\forall x\in\mathbb{R}\setminus\lbrace0\rbrace
\end{equation*} 
and $f(0)=\sigma(0)=w(0)=1$.

By \hyperref[prelest]{Lemma \ref*{prelest}} we have $[w,\sigma]_{\ci{A_{p}}}\sim_{p}(1-\beta)^{-1}$, $[\sigma]_{\ci{A_{\infty}}}\lesssim_{p}1$ and $[w]_{\ci{A_{\infty}}}\leq B(1-\beta)^{-1}$, for some absolute constant $B>0$. We also have
\begin{eqnarray*}
||f||_{\ci{L^{p}(\sigma)}}^{p}=\int_{[0,1)}x^{-\beta}dx=(1-\beta)^{-1}.
\end{eqnarray*}

\begin{lemma}
There holds $||S(\sigma f)||_{\ci{L^{p}(w)}}\gtrsim_{p}(1-\beta)^{-\frac{1}{2}-\frac{1}{p}}$.
\end{lemma}
\begin{proof}
Set $I_{n}=\left[0,\frac{1}{2^{n}}\right),\;n=0,1,2,\ldots$ and $J_{n}=I_{n-1}\setminus I_{n}=\left[\frac{1}{2^{n}},\frac{1}{2^{n-1}}\right),\;n=1,2,\ldots$. We notice that 
\begin{align*}
\langle f\sigma\rangle_{I_{k}}=\frac{1}{|I_{k}|}\int_{I_{k}}f\sigma dx=2^{k}\sum_{n=k+1}^{\infty}\int_{J_{n}}(-1)^{n-1}dx=-2^{k}\sum_{n=k+1}^{\infty}\left(-\frac{1}{2}\right)^{n}dx=\frac{(-1)^{k}}{3},
\end{align*}
for all $k\geq0$. It follows that for all $k\geq 0$, for all $x\in I_{k+1}$, there holds
\begin{align*}
\Delta_{\ci{I_{k}}}(\sigma f)(x)&=\langle \sigma f\rangle_{\ci{I_{k+1}}}-\langle\sigma f\rangle_{\ci{I_{k}}}=\frac{2(-1)^{k+1}}{3},
\end{align*}
therefore $|\Delta_{\ci{I_{k}}}(\sigma f)(x)|\sim1$. In particular, for all $k\geq 0$, for all $n>k$ and for all $x\in I_{n}$, we have $|\Delta_{\ci{I_{k}}}(\sigma f)(x)|\sim1$. It follows that for all $n\geq1$, for all $x\in J_{n}$, there holds
\begin{eqnarray*}
S(\sigma f)(x)^{p}\geq\left(\sum_{k=0}^{n-1}(\Delta_{\ci{I_{k}}}(\sigma f)(x))^2\right)^{\frac{p}{2}}\sim_{p}\left(\sum_{k=2}^{n+1}1\right)^{\frac{p}{2}}= n^{\frac{p}{2}}.
\end{eqnarray*}
 Moreover, for all $n\geq 1$, we have $w(J_{n})\sim 2^{(\beta-1)n}$. It follows that
\begin{equation*}
||S(\sigma f)||_{\ci{L^{p}(w)}}^{p}\gtrsim_{p}\sum_{n=1}^{\infty}2^{(\beta-1)n}n^{\frac{p}{2}}.
\end{equation*}
It is now easily seen that for all $n\geq 1$ and for $x\in[n,n+1]$, there holds
\begin{equation*}
\frac{1}{2}2^{(\beta-1)n}n^{\frac{p}{2}}\leq2^{(\beta-1)x}x^{\frac{p}{2}}\leq 2^{\frac{p}{2}}2^{(\beta-1)n}n^{\frac{p}{2}}.
\end{equation*}
It follows that
\begin{align*}
\sum_{n=1}^{\infty}2^{(\beta-1)n}n^{\frac{p}{2}}&\sim_{p} \int_{(1,\infty)}2^{(\beta-1)x}x^{\frac{p}{2}}dx
=(1-\beta)^{-\frac{p}{2}-1}\int_{(1-\beta,\infty)}2^{-y}y^{\frac{p}{2}}dy\\
&\geq (1-\beta)^{-\frac{p}{2}-1}\int_{(1,\infty)}2^{-y}y^{\frac{p}{2}}dy \sim_{p}
(1-\beta)^{-\frac{p}{2}-1}.
\end{align*}
It follows that $||S(\sigma f)||_{\ci{L^{p}(w)}}^{p}\gtrsim_{p}(1-\beta)^{-\frac{p}{2}-1}$.
\end{proof}

Note that in particular $||S(\cdot\sigma)||_{\ci{L^{p}(\sigma)\rightarrow L^{p}(w)}}\geq\frac{||S(\sigma f)||_{\ci{L^{p}(w)}}}{||f||_{\ci{L^{p}(\sigma)}}}\gtrsim_{p}(1-\beta)^{-\frac{1}{2}}$ (and this remains true even if $S$ is just the square function over the unit interval $[0,1)$, since we used only dyadic intervals contained in $[0,1)$ in the proof of Lemma 4.1. above and $f$ is supported in $[0,1]$).

The above estimates yield
\begin{equation*}
(1-\beta)^{-\frac{1}{2}-\frac{1}{p}}\lesssim_{p}(1-\beta)^{-\frac{1}{p}}\phi(B(1-\beta)^{-1})(1-\beta)^{-\frac{1}{p}},
\end{equation*}
i.e.
\begin{equation*}
\phi(B(1-\beta)^{-1})\gtrsim_{p} (1-\beta)^{-\frac{1}{2}+\frac{1}{p}}.
\end{equation*}
Since $\beta\in(0,1)$ was arbitrary, we obtain $\phi(x)\gtrsim_{p} x^{\frac{1}{2}-\frac{1}{p}}$, for all $x>B$.
\newline

\begin{remark}
It is worth noting that another way to prove this estimate passes through the family of examples constructed by Lacey--Scurry in \cite{Lacey-Scurry} Section 3. and further refined by Hyt$\ddot{\text{o}}$nen--Li in \cite{Hyt-Li} Proposition 5.2., in exactly the same way that Lerner's family of examples yielded the estimate for the exponent of $[\sigma]_{A_{\infty}}$. However, the construction by Lacey--Scurry and Hyt$\ddot{\text{o}}$nen--Li views the square function as a vector-valued linear operator and then makes use of a  duality argument and Khinchine's inequalities, whereas in our construction above we worked directly with the square function (in particular, we avoided Khinchine's inequalities by making an appropriate choice of signs).
\newline

\end{remark}

\section{Counterexamples}

\subsection{Direct sum of singularities}

In this section, we provide one counterexample proving \hyperref[countprop]{Proposition \ref*{countprop}}. As mentioned in the introduction, this counterexample can be viewed as a modified version of a direct sum construction over the family of examples in \hyperref[expw]{Section \ref*{expw}}. Here, by direct sum construction one understands that an interval is divided into subintervals $I_1,I_2,\ldots$, and then each $I_{k}$ is equipped with an (appropriately shifted and rescaled) example from Section 4. for some value $\beta_{k}$ of the parameter $\beta$, such that $\lim_{k\rightarrow\infty}\beta_{k}=1$. This direct sum construction is possible whenever one has uniform control over the joint $A_p$-characterictic of the two weights, thus one will have to slightly modify the family of examples in \hyperref[expw]{Section \ref*{expw}} (see (5.1) below). I am grateful to Professor S. Treil for suggesting this direct sum construction. In general, this direct sum construction will only ensure that the dyadic joint $A_{p}$-characteristic of the two weights is finite. In order to make sure that the full joint $A_{p}$-characteristic of the constructed weights is finite, we will use a modified version of the direct sum construction.

Let $p\in(2,\infty)$. Set $\beta_{k}=1-\frac{1}{k}$. Let $k\geq 1$ be arbitrary. Consider the weights $w_{k},\sigma_{k}$ on $[0,1)$ given by
\begin{equation}
w_{k}(x)=(1-\beta_{k})x^{-\beta_k},\;\sigma_{k}=x^{\frac{\beta_{k}}{p-1}},\;\forall x\in(0,1)
\end{equation}
and $w_{k}(0)=\sigma_k(0)=1$. Then $w_{k}([0,1))\sim_{p}1$ and $\sigma_{k}([0,1))\sim_{p}1$. Moreover, by \hyperref[prelest]{Lemma \ref*{prelest}} we have $[w_k,\sigma_k]_{\ci{A_{p}}}\lesssim_{p}1$ and $[\sigma_{k}]_{\ci{A_{q'}}}\lesssim_{p}1$, where $q=\frac{p+1}{2}$. I am grateful to Professor S. Treil for suggesting this modified version of the family of examples of \hyperref[expw]{Section \ref*{expw}}.

We can write $Q_0\setminus\lbrace0\rbrace=\bigcup_{k=1}^{\infty}J_{k}$, where
\begin{equation*}
J_{k}=\left[\frac{1}{2^{k}},\frac{1}{2^{k-1}}\right),\;k=1,2,\ldots.
\end{equation*}
Set also $I_{k}=\left[0,\frac{1}{2^{k}}\right)=\bigcup_{n=k+1}^{\infty}J_{n}$, for $k=0,1,2,\ldots$. Now, for all $k=1,2,\ldots$, one can consider the weights $\widetilde{w}_{k},\widetilde{\sigma}_{k}$ on $J_k$ given by
\begin{equation*}
\widetilde{w}_{k}(x)=
\begin{cases}
w_{k}(2^{k}x-1),\text{ if }k\text{ is odd}\\
w_{k-1}(2-2^{k}x),\text{ if }k\text{ is even}
\end{cases}
,\forall x\in \text{Int}(J_{k}),
\end{equation*}
\begin{equation*}
\widetilde{\sigma}_{k}(x)=
\begin{cases}
\sigma_{k}(2^{k}x-1),\text{ if }k\text{ is odd}\\
\sigma_{k-1}(2-2^{k}x),\text{ if }k\text{ is even}
\end{cases}
,\forall x\in \text{Int}(J_{k}),
\end{equation*}
and $\widetilde{w}_{k}\left(\frac{1}{2^{k}}\right)=\widetilde{\sigma}_{k}\left(\frac{1}{2^{k}}\right)=1$. Consider the weights $\widetilde{w},\widetilde{\sigma}$ on $Q_0$ given by
\begin{equation*}
\widetilde{w}(x)=\widetilde{w}_{k}(x),\;\;\;\widetilde{\sigma}(x)=\widetilde{\sigma}_k(x),\;\forall x\in J_k,\;\forall k=1,2,\ldots,
\end{equation*}
and $\widetilde{w}(0)=\widetilde{\sigma}(0)=1$. Following the same steps as in the proof of \hyperref[glob]{Lemma \ref*{glob}}, essentially with $\mathbb{Z}$ replaced by $\lbrace0\rbrace\cup\left\lbrace\frac{1}{2^{k}}:\;k\in\mathbb{N}\right\rbrace$, it can be verified that $[\widetilde{w},\widetilde\sigma]_{\ci{A_{p}}}<\infty$ and $[\widetilde\sigma]_{\ci{A_{q'}}}<\infty$.

Let now $S_k$ be the square function over $J_k$, for all $k=1,2,\ldots$. Then, the computations in section 4. coupled with translation and rescaling invariance show that
\begin{equation*}
||S_{k}(\cdot\widetilde{\sigma}_{k})||_{\ci{L^{p}(\widetilde{\sigma}_k)\rightarrow L^{p}(\widetilde{w}_{k})}}\gtrsim_{p}(1-\beta_{k})^{\frac{1}{p}-\frac{1}{2}},
\end{equation*}
for all odd positive integers $k$. In particular, for all odd positive integers $k$, there exists $f_{k}\in L^{p}(\widetilde{\sigma}_{k})$ with $||f_{k}||_{\ci{L^{p}(\widetilde{\sigma}_{k})}}=\frac{1}{k^{\frac{1}{2}}}$ and $||S_{k}(f_{k}\widetilde{\sigma}_{k})||_{\ci{L^{p}(\widetilde{w}_{k})}}\gtrsim_{p}\frac{(1-\beta_{k})^{\frac{1}{p}-\frac{1}{2}}}{k^{\frac{1}{2}}}=\frac{1}{k^{\frac{1}{p}}}$. Consider the function $f$ on $[0,1)$ given by
\begin{equation*}
f(x)=f_{k}(x),\;\forall x\in J_{k},
\end{equation*}
for all odd positive integers $k$, and $f(x)=0$, for all $x\in[0,1)\setminus\left(\bigcup_{k=0}^{\infty}J_{2k+1}\right)$. Then, we have
\begin{equation*}
||f||_{\ci{L^{p}(\widetilde{\sigma})}}^{p}=\sum_{k=0}^{\infty}||f_{2k+1}||_{\ci{L^{p}(\widetilde{\sigma}_{2k+1})}}^{p}=\sum_{k=0}^{\infty}\frac{1}{(2k+1)^{\frac{p}{2}}}<\infty,
\end{equation*}
therefore $f\in L^{p}(\widetilde{\sigma})$, and
\begin{equation*}
||S(f\widetilde{\sigma})||_{\ci{L^{p}(\widetilde{w})}}^{p}\geq\sum_{k=0}^{\infty}||S_{2k+1}(f_{2k+1}\widetilde{\sigma}_{2k+1})||_{\ci{L^{p}(\widetilde{w}_{2k+1})}}^{p}\gtrsim_{p}\sum_{k=0}^{\infty}\frac{1}{2k+1}=\infty,
\end{equation*}
therefore $S(f\widetilde{\sigma})\notin L^{p}(w)$ (here $S$ denotes the square function over $[0,1)$), concluding the proof.

Note that by \hyperref[glob]{Lemma \ref*{glob}} (with $x_0=\frac{3}{4}$) we have that there exist weights $\widetilde{w}',\widetilde{\sigma}'$ on $\mathbb{R}$ with $\widetilde{w}'|_{[0,1)}=\widetilde{w}$ and $\widetilde{\sigma}'|_{[0,1)}=\widetilde\sigma$, such that $[\widetilde{w}',\widetilde{\sigma}']_{\ci{A_{p}}}\lesssim_{p}1$ and $[\widetilde{\sigma}']_{\ci{A_{q'}}}\lesssim_{p}1$. Therefore, the above counterexample can be viewed as a counterexample on $\mathbb{R}$ (the square function over $\mathbb{R}$ clearly dominates the one over $[0,1)$).
\newline

\begin{remark}
Assume that we repeat the above construction, but with $(\beta_{k})_{k=1}^{\infty}$ being an arbitrary strictly increasing sequence of elements of $[0,1)$ with $\lim_{k\rightarrow\infty}\beta_{k}=1$. Suppose that $S(f\tilde{\sigma})\in L^{p}(\tilde{w})$, for all $f\in L^{p}(\tilde{\sigma})$ (here $S$ denotes the square function over $[0,1)$). Then, since $\frac{1}{p}-\frac{1}{2}<0$ and
\begin{equation*}
||S(\cdot\widetilde\sigma)||_{\ci{L^{p}(\widetilde\sigma)\rightarrow L^{p}(\widetilde w)}}\geq||S_{2k+1}(\cdot\widetilde{\sigma}_{2k+1})||_{\ci{L^{p}(\widetilde{\sigma}_{2k+1})\rightarrow L^{p}(\widetilde{w}_{2k+1})}},\;\forall k=0,1,2,\ldots,
\end{equation*}
we obtain $||S(\cdot\widetilde\sigma)||_{\ci{L^{p}(
\widetilde\sigma)\rightarrow L^{p}(\widetilde w)}}=\infty$. Then, viewing $S(\cdot\widetilde{\sigma})$ as a linear operator from $L^{p}(\widetilde{\sigma})$ into $L^{p}(l^{2}(\mathcal{D}),\widetilde{w})$ (where $\mathcal{D}$ is the family of all dyadic subintervals of $[0,1)$), we obtain a contradiction through an easy application of the closed graph theorem. Therefore, there exists $f\in L^{p}(\tilde{\sigma})$ with $S(f\tilde{\sigma})\notin L^{p}(\tilde{w})$.
\newline

\end{remark}
\subsection{Following the strategy of Lai--Treil \cite{Lai-Treil}}

In this section, we provide a second counterexample proving \hyperref[countprop]{Proposition \ref*{countprop}}. As mentioned in the introduction, the construction follows very closely the strategy of the one in \cite{Lai-Treil} Section 3.

Let $p\in(2,\infty)$. As in \cite{Lai-Treil} Section 3., choose $r\in\left(\frac{1}{p},\frac{1}{2}\right)$. Set $\alpha=2r-1\in(-1,0)$ and consider the functions $f,\sigma,w:[0,1)\rightarrow[0,\infty)$ given by
\begin{equation*}
f(x)=\frac{(-1)^{\lfloor-\log_2 x\rfloor}}{x^{\frac{1}{p-1}}{(1-\log_2 x)^{r}}},\;\;\sigma(x)=x^{\frac{1}{p-1}},\;\;w(x)=\frac{1}{x(1-\log_2 x)^{1-\alpha}},\;\forall x\in(0,1)
\end{equation*} 
and $f(0)=\sigma(0)=w(0)=1$.

First of all, we claim that $[w,\sigma]_{\ci{A_{p}}}<\infty$. Indeed, let $a,b\in[0,1]$ with $a<b$ be arbitrary. We distinguish the following cases:

(A) There holds $a=0$. Then, by direct computation we have
\begin{eqnarray*}
\langle w\rangle_{\ci{[a,b)}}\langle\sigma\rangle_{\ci{[a,b)}}^{p-1}
=-\frac{\log 2}{\alpha}\left(\frac{p-1}{p}\right)^{p-1}(1-\log_{2} b)^{\alpha}\leq -\frac{\log 2}{\alpha}\left(\frac{p-1}{p}\right)^{p-1}<\infty,
\end{eqnarray*}
since $1-\log_{2}b\geq 1$ and $\alpha<0$.

(B) There holds $a>0$ and $a\leq\frac{b}{2}$. Then, we have $[a,b)\subseteq[0,b)$ and $b\leq 2(b-a)$, therefore by case (A) we obtain
\begin{eqnarray*}
\langle w\rangle_{\ci{[a,b)}}\langle\sigma\rangle_{\ci{[a,b)}}^{p-1}\leq
2\langle w\rangle_{\ci{[0,b)}}(2\langle\sigma\rangle_{\ci{[0,b)}})^{p-1}\leq
-\frac{2^{p}\log 2}{\alpha}\left(\frac{p-1}{p}\right)^{p-1}<\infty.
\end{eqnarray*}

(C) There holds $a>0$ and $a>\frac{b}{2}$. It is easy to see that
\begin{equation*}
\sup_{x\in[a,b)}(1-\log_{2}x)=1-\log_{2}a\leq2(1-\log_{2}b)=2\inf_{x\in[a,b)}(1-\log_{2}x),
\end{equation*}
which can be easily seen to imply that there exists some constant $C>1$ (depending only on $\alpha$) such that
\begin{equation*}
\sup_{x\in[a,b)}\frac{1}{x(1-\log_{2}x)^{1-\alpha}}\leq C \inf_{x\in[a,b)}\frac{1}{x(1-\log_{2}x)^{1-\alpha}}.
\end{equation*}
It follows that
\begin{align*}
\langle w\rangle_{\ci{[a,b)}}\langle\sigma\rangle_{\ci{[a,b)}}^{p-1}&\leq\left(\sup_{x\in[a,b)}\frac{1}{x(1-\log_{2}x)^{1-\alpha}}\right)(\sup_{x\in(a,b)}x^{\frac{1}{p-1}})^{p-1}\\
&\leq\left(C\inf_{x\in[a,b)}\frac{1}{x(1-\log_{2}x)^{1-\alpha}}\right)(2^{\frac{1}{p-1}}\inf_{x\in(a,b)}x^{\frac{1}{p-1}})^{p-1}\\
&\leq2C\inf_{x\in[a,b)}\frac{x}{x(1-\log_{2}x)^{1-\alpha}}
=\frac{2C}{(1-\log_{2}a)^{1-\alpha}}\leq 2C,
\end{align*}
since $1-\alpha>0$ and $\log_{2}b\leq 0$.

Similar estimates show that $[\sigma]_{\ci{A_{2}}}<\infty$, therefore in particular $[\sigma]_{\ci{A_{\infty}}}<\infty$.

We also claim that $f\in L^{p}(\sigma)$. Indeed, by direct computation we have
\begin{eqnarray*}
\int_{[0,1)}|f|^{p}\sigma dx=\int_{(0,1)}\frac{1}{x(1-\log_2 x)^{pr}}dx=\frac{\log 2}{pr-1}<\infty,
\end{eqnarray*}
since $pr>1$.

Set now $I_{k}=\left[0,\frac{1}{2^{k}}\right)$, for all $k=0,1,2,\ldots$. Let $k\geq 0$ be arbitrary. We notice that for all $x\in I_{k+1}$, there holds
\begin{align}
-\Delta_{\ci{I_{k}}}(\sigma f)(x)
&=(-2)^{k}\int_{\left(\frac{1}{2^{k+1}},\frac{1}{2^{k}}\right)}\frac{1}{(1-\log_2 x)^{r}}dx-
2^{k}\int_{\left(0,\frac{1}{2^{k+1}}\right)}\frac{(-1)^{\lfloor-\log_{2}x\rfloor}}{(1-\log_2 x)^{r}}dx.
\end{align}
We have
\begin{eqnarray*}
\int_{\left(0,\frac{1}{2^{k+1}}\right)}\frac{(-1)^{\lfloor-\log_{2}x\rfloor}}{(1-\log_2 x)^{r}}dx=\sum_{l=k+1}^{\infty}(-1)^{l}\int_{\left(\frac{1}{2^{l+1}},\frac{1}{2^{l}}\right)}\frac{1}{(1-\log_2 x)^{r}}dx,
\end{eqnarray*}
and for all $l\geq 0$ there holds
\begin{equation}
\frac{1}{2^{l+1}(l+2)^{r}}\leq
\int_{\left(\frac{1}{2^{l+1}},\frac{1}{2^{l}}\right)}\frac{1}{(1-\log_2 x)^{r}}dx\leq\frac{1}{2^{l+1}(l+1)^{r}},
\end{equation}
therefore by the alternating series test we obtain
\begin{equation}
\left|\int_{\left(0,\frac{1}{2^{k+1}}\right)}\frac{(-1)^{\lfloor-\log_{2}x\rfloor}}{(1-\log_2 x)^{r}}dx\right|\leq \int_{\left(\frac{1}{2^{k+2}},\frac{1}{2^{k+1}}\right)}\frac{1}{(1-\log_2 x)^{r}}dx\leq
\frac{1}{2^{k+2}(k+2)^{r}}.
\end{equation}
From (5.2), (5.3) and (5.4) we deduce $|\Delta_{\ci{I_{k}}}(\sigma f)(x)|\geq\frac{1}{4(k+2)^{r}}$, for all $x\in I_{k}$. This estimate is a modification of the first estimate on page 11 in \cite{Lai-Treil}. In particular, for all $k\geq 0$, for all $n>k$ and for all $x\in I_{n}$, we have $|\Delta_{\ci{I_{k}}}(\sigma f)(x)|\geq\frac{1}{4(k+2)^{r}}$. It follows that for all $k\geq1$, for all $x\in I_{k}$, there holds
\begin{equation*}
S(\sigma f)(x)^{p}\geq\left(\sum_{n=0}^{k-1}(\Delta_{\ci{I_{n}}}(\sigma f)(x))^2\right)^{\frac{p}{2}}\geq C_{1}\left(\sum_{n=2}^{k+1}\frac{1}{n^{2r}}\right)^{\frac{p}{2}}=C_1\left(\sum_{n=2}^{k+1}\frac{1}{n^{1+\alpha}}\right)^{\frac{p}{2}},
\end{equation*}
where $C_1=\frac{1}{4^{p}}>0$ (here $S$ denotes the square function over $[0,1)$). This estimate is the same as the second estimate on page 11 in \cite{Lai-Treil}. Moreover, for all $k\geq 1$, we have
\begin{eqnarray*}
w(I_{k})=\int_{\left(0,\frac{1}{2^{k}}\right)}\frac{1}{x(1-\log_2 x)^{1-\alpha}}dx=-\frac{1}{\alpha}(k+1)^{\alpha},
\end{eqnarray*}
and
\begin{equation*}
\left(\sum_{n=2}^{k+1}\frac{1}{n^{1+\alpha}}\right)^{-1}\leq\left(\int_{[2,k+2]}\frac{1}{x^{1+\alpha}dx}\right)^{-1}
-\alpha\frac{2^{\alpha}(k+2)^{\alpha}}{{2^\alpha}-(k+2)^{\alpha}}\leq-\alpha\frac{2^{\alpha}(k+1)^{\alpha}}{2^{\alpha}-3^{\alpha}},
\end{equation*}
since $\alpha<0$, therefore
\begin{equation*}
w(I_{k})\geq C_2\left(\sum_{n=2}^{k+1}\frac{1}{n^{1+\alpha}}\right)^{-1},
\end{equation*}
where $C_2=\frac{2^{\alpha}-3^{\alpha}}{2^{\alpha+1}\alpha^2}>0$. It follows that
\begin{equation*}
\int_{I_{k}}S(\sigma f)(x)^pw(x)dx\geq C_1C_2\left(\sum_{n=2}^{k+1}\frac{1}{n^{1+\alpha}}\right)^{\frac{p}{2}-1},\;\forall k\geq 1,
\end{equation*}
therefore $\lim_{k\rightarrow\infty}\int_{I_{k}}S(\sigma f)(x)^pw(x)dx=\infty$, since $1+\alpha=2r<1$ and $\frac{p}{2}-1>0$. This estimate is a slight modification of the third estimate on page 11 in \cite{Lai-Treil}. It follows that $S(\sigma f)\notin L^{p}(w)$.

Note that by \hyperref[glob]{Lemma \ref*{glob}} (with $x_0=\frac{1}{2}$) we have that there exist weights $\widetilde{w},\widetilde{\sigma}$ on $\mathbb{R}$ with $\widetilde{w}|_{[0,1)}=w$ and $\widetilde{\sigma}|_{[0,1)}=\sigma$, such that $[\widetilde{w},\widetilde{\sigma}]_{\ci{A_{p}}}\lesssim_{p}1$ and $[\widetilde{\sigma}]_{\ci{A_{2}}}\lesssim_{p}1$. Therefore, the above counterexample can be viewed as a counterexample on $\mathbb{R}$ (the square function over $\mathbb{R}$ clearly dominates the one over $[0,1)$).
\newline

\end{document}